\documentclass[11pt,reqno]{amsart}

\usepackage[utf8]{inputenc}
\usepackage[T1]{fontenc}
\usepackage[english]{babel}
\usepackage{amssymb}
\usepackage{graphicx}
\usepackage[a4paper,margin=30mm]{geometry}
\usepackage[hidelinks]{hyperref}
\hypersetup{
  pdftitle={The Samuelson Condition and Tangent Lines of Quadratic Curves},
  pdfauthor={Yasuhiro Kurokawa},
  pdfsubject={Preprint},
  pdfkeywords={Web geometry, Samuelson condition, Lagrangian 2-web, Quadratic curves, Envelope}
}

\newcommand{\R}{\mathbb{R}}

\theoremstyle{plain}
\newtheorem{theorem}{Theorem}[section]
\newtheorem{proposition}[theorem]{Proposition}
\theoremstyle{remark}
\newtheorem{remark}[theorem]{Remark}

\begin{document}

\title[The Samuelson Condition and Quadratic Curves]{The Samuelson Condition and Tangent Lines of Quadratic Curves}

\author{Yasuhiro Kurokawa}
\address{Department of Architecture, School of Architecture, Shibaura Institute of Technology, Koto-ku, Toyosu 3-7-5, Tokyo, Japan}
\email{kurokawa@sic.shibaura-it.ac.jp}
\thanks{ORCID: 0009-0004-4872-7388}

\date{}

\begin{abstract}
We show that the Samuelson condition, an area condition, is not satisfied by the Lagrangian 2-web formed by the tangent lines to a non-degenerate real quadratic curve.
\end{abstract}

\maketitle

\section{Introduction}

Let $(\R^2, \omega)$ be the symplectic plane with the canonical symplectic form $\omega=dx\wedge dy.$
A pair $(\mathcal F_1,\, \mathcal F_2)$ of Lagrangian foliations of codimension 1 in $\R^2$, such that the leaves of $\mathcal F_1$ are transverse to the leaves of $\mathcal F_2$, is called a \textit{Lagrangian 2-web} in $\R^2$.
Two Lagrangian 2-webs are said to be equivalent if there exists a symplectic diffeomorphism mapping one to the other while preserving the foliations.

In \cite{CooperRussellSamuelson}, the authors discussed an area condition associated with the extremal properties of equilibrium systems which arise in the context of economic agents and classical thermodynamics, originally characterized by Samuelson.
The condition is described as follows. Consider the four areas $A, B, C, D$ of quadrilaterals formed by the curves of a Lagrangian 2-web as shown in Figure~\ref{fig:samuelson}.

\begin{figure}[h]
\begin{center}
\includegraphics[width=1.0\textwidth, trim=-6cm 4cm 9cm 0cm, clip]{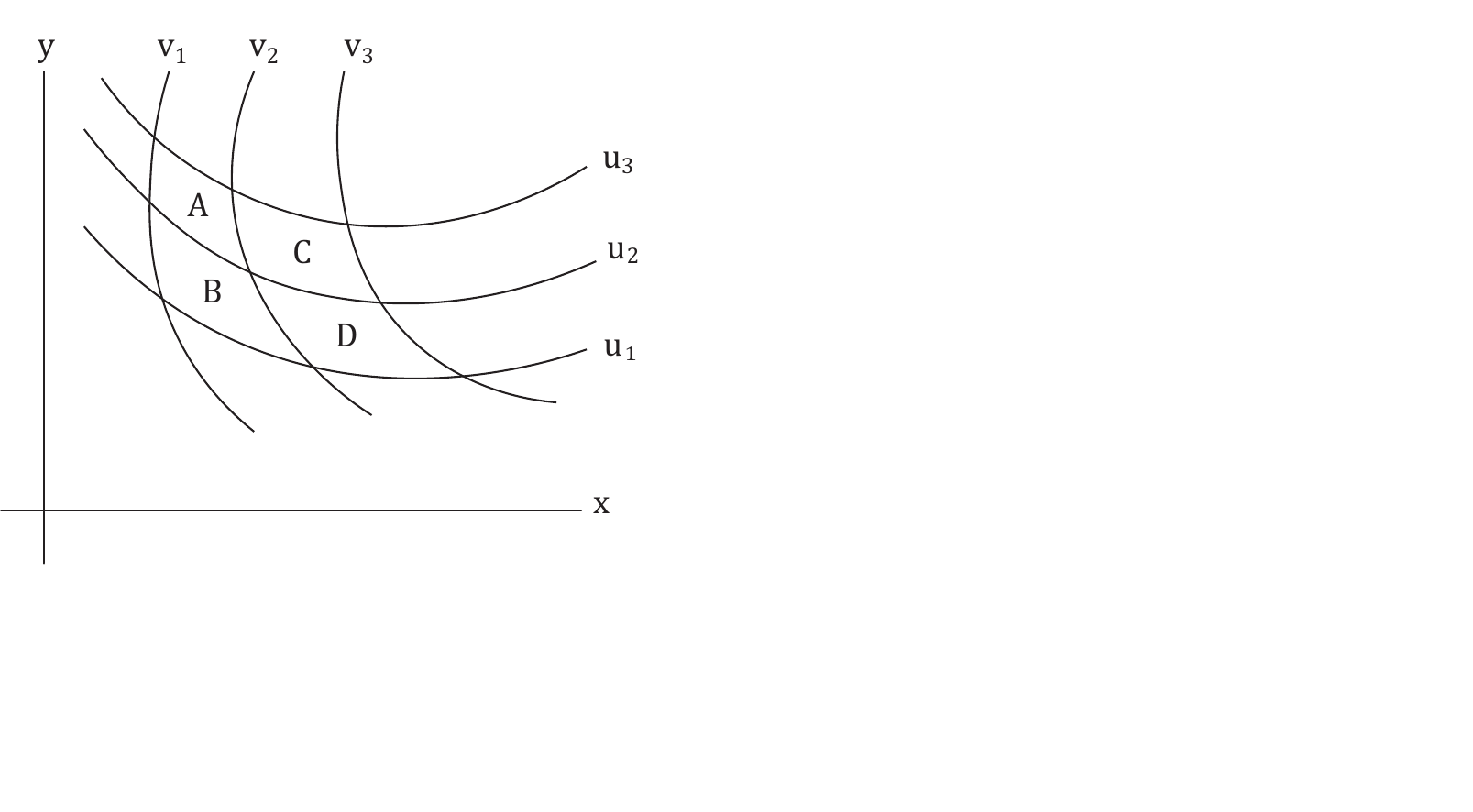}
\caption{Samuelson condition}
\label{fig:samuelson}
\end{center}
\end{figure}

Consider the equation
\begin{equation*}
\frac{A}{B}=\frac{C}{D}.
\end{equation*}
If this equation holds for any four quadrilaterals formed by the curves as in Figure \ref{fig:samuelson}, then the Lagrangian 2-web is said to satisfy the \textit{Samuelson condition}. We call such a web a \textit{Samuelson web}.
The standard Samuelson web is given by $x=\text{const.},\,y=\text{const.}$ in $\R^2=\{(x, y)\}.$

As mentioned in \cite{Tabachnikov} (p.273), it is an intriguing problem to characterize the families of tangents to curves that are equivalent to the standard 2-web. In the case of 3-webs formed by line families in the plane, the Theorem of Graf and Sauer provides a beautiful characterization \cite{Blaschke, Chern, PereiraPirio}: such a web is hexagonal if and only if it consists of the family of tangents to an algebraic curve of class three.

Drawing motivation from this classic result, we investigate, as a first step in this paper, whether the family of tangents to a non-degenerate quadratic curve forms a Samuelson web. Our analysis reveals that these 2-webs, 
despite being the fundamental algebraic cases arising from tangent lines, fail to satisfy the Samuelson condition. By establishing this incompatibility within the algebraic framework, we highlight the necessity of turning our attention to transcendental curves in future research to identify geometric structures that satisfy this area condition.

This paper is organized as follows. In Section 2, we review the definition of Samuelson webs and present a simplified formula for the d'Alembert condition. Sections 3 and 4 are devoted to the description of tangent lines to quadratic curves and the proof of the main theorem.

Throughout this paper, we assume that all mappings are of class $C^{\infty}$, unless otherwise stated.

\section{Samuelson Webs and the Area Condition}
In this section, we review the definition of Samuelson webs and the associated area condition as discussed in \cite{CooperRussellSamuelson, FerraraUdriste, OzawaSatoSuzuki}.

Let
$$u=u(x,y), \quad v=v(x,y)$$
be two smooth functions on $\R^2=\{(x, y)\}$ such that their level curves are transverse.
These functions define a 2-web $\mathcal W$ on $\R^2$.

By a foliation-preserving diffeomorphism
$$\phi : (x, y)\mapsto (u(x, y), v(x, y)),$$
$\mathcal W$ is transformed into the standard 2-web formed by $u=\text{const.}, v=\text{const.}$ on $\R^2=\{(u, v)\}.$
However, $\phi$ is not necessarily area-preserving, and hence not necessarily a symplectic diffeomorphism.

The Jacobian of $\phi$ is given by
$$J_\phi(x,y)=u_x(x,y)v_y(x,y)-u_y(x,y)v_x(x,y).$$
The Jacobian of $\phi^{-1}$ is given by
$$J_{\phi^{-1}}(u,v)=(u_x(x,y)v_y(x,y)-u_y(x,y)v_x(x,y))^{-1}.$$

Replacing $u$ and $v$ with $U(u)$ and $V(v)$, respectively, the diffeomorphism $R : (u, v)\mapsto (U(u),\,V(v))$ that preserves the standard web is called a \textit{recalibration}.

\begin{proposition}[\cite{CooperRussellSamuelson}]
Let $\mathcal W$ be a Lagrangian 2-web on $\R^2$ defined by level functions $u=u(x,y), v=v(x,y).$
Then the following are equivalent:
\begin{enumerate}
    \item[(1)] $\mathcal W$ is a Samuelson web.
    \item[(2)] The Jacobian $J_{\phi^{-1}}(u, v)$ factorizes as
    $$J_{\phi^{-1}}(u,v)=a(u) \cdot b(v).$$
    \item[(3)] There exists a recalibration $R: (u, v)\mapsto (U(u), V(v))$ such that
    $$J_{R\circ \phi}(x, y)=1.$$
\end{enumerate}
\end{proposition}

\begin{remark}[\cite{CooperRussellSamuelson, OzawaSatoSuzuki}]
If $J_{\phi^{-1}}$ splits multiplicatively as $J_{\phi^{-1}}=a(u) \cdot b(v)$, then $\log J_{\phi^{-1}}=\log a(u) + \log b(v)$ satisfies the d'Alembert condition \cite{PrastaroRassias}:
$$\frac{\partial^2\log J_{\phi^{-1}}}{\partial u\partial v}=0.$$
\end{remark}

Next, we calculate $\frac{\partial^2\log J_{\phi^{-1}}}{\partial u\partial v}$ following \cite{OzawaSatoSuzuki}.

Let $a(x,y), b(x,y)$ be two smooth functions on $\R^2=\{(x, y)\}.$
Define a 2-web on $\R^2$ by the differential equations
$$\frac{dy}{dx}=a(x,y), \quad \frac{dy}{dx}=b(x, y).$$
The leaves are integral curves of the vector fields
$$\frac{\partial}{\partial x}+a(x,y)\frac{\partial}{\partial y}, \quad
\frac{\partial}{\partial x}+b(x,y)\frac{\partial}{\partial y}.$$
They are level curves of the first integrals of the above differential equations.
The first integrals $u(x, y), \, v(x, y)$ satisfy
\begin{equation}
\left(\frac{\partial}{\partial x}+a(x,y)\frac{\partial}{\partial y}\right)u=0, \quad
\left(\frac{\partial}{\partial x}+b(x,y)\frac{\partial}{\partial y}\right)v=0.
\label{eq:first_integrals}
\end{equation}

From this, we obtain
\begin{align*}
J_{\phi^{-1}} &= ( (-a u_y)v_y -(-b v_y) u_y)^{-1} = (u_yv_y(b-a))^{-1}.
\end{align*}

Furthermore, by \eqref{eq:first_integrals}, we have the change of basis:
\begin{align*}
\frac{\partial}{\partial u} &= x_u\frac{\partial}{\partial x}+y_u\frac{\partial}{\partial y} = J_{\phi^{-1}}\left(v_y\frac{\partial}{\partial x}-v_x\frac{\partial}{\partial y}\right)\\
&= (u_y(b-a))^{-1}\left(\frac{\partial}{\partial x}+ b\frac{\partial}{\partial y}\right),\\
\frac{\partial}{\partial v} &= x_v\frac{\partial}{\partial x}+y_v\frac{\partial}{\partial y} = J_{\phi^{-1}}\left(-u_y\frac{\partial}{\partial x}+u_x\frac{\partial}{\partial y}\right)\\
&= -(v_y(b-a))^{-1}\left(\frac{\partial}{\partial x}+ a\frac{\partial}{\partial y}\right).
\end{align*}
Then we differentiate $\log J_{\phi^{-1}}$ with respect to $u$ and $v$. In \cite{OzawaSatoSuzuki}, the following identity was derived with the help of MAPLE.

\begin{theorem}[\cite{OzawaSatoSuzuki}]
The term related to the d'Alembert condition is given by:
{\small
\begin{equation}
\begin{split}
{u_y}^2 {v_y}^2 (a-b)^4 \frac{\partial^2\log J_{\phi^{-1}}}{\partial u\partial v}
&= a_{xx}a + 2ba_{xy} a - 2a^2 b_{xy} - 2a^2{b_y}^2 - 2b^2{a_y}^2 \\
&\quad + a_x b_y b + 4b_xa_x + a^2b_{yy} b - 4ab_y b_x - b^3 a_{yy} \\
&\quad - 2{b_x}^2 - 2{a_x}^2 + 3b_xa_y b + b_x a_y a + ab^2 a_{yy} \\
&\quad - 4a_x a_yb - a^3 b_{yy} + 2ab b_{xy} - b_{xx}a + 3a_xb_y a \\
&\quad - 2b^2 a_{xy} + b_{xx} b - a_{xx}b + 4ab_y a_y b.
\end{split}
\label{eq:big_equation}
\end{equation}
}
\end{theorem}

As mentioned in \cite{OzawaSatoSuzuki}, the above equation in Theorem 2.3 is essentially the formula in \cite{CooperRussell2006}, up to the placement of parentheses.
Furthermore, a similar recalculation, called the Holy Grail Equation in \cite{CooperRussell2009}, appears to be incorrect, as noted in \cite{OzawaSatoSuzuki}.

\begin{remark}
Let us denote the right-hand side of equation \eqref{eq:big_equation} in Theorem 2.3 by $K$. It can be reorganized as follows:
\begin{gather*}
K = 3(a_x+ b a_y )(b_x+a b_y )+(a_x+a a_y )(b_x+ b b_y )
-2(a_x+b a_y)^2-2(b_x+a b_y)^2\\
+(a-b)(a_{xx}+2ba_{xy}+b^2a_{yy})
+(b-a)(b_{xx}+2ab_{xy}+a^2b_{yy}).
\end{gather*}
\end{remark}

\section{Quadratic Curves and Their Tangent Lines}

We regard a quadratic curve as the envelope of a family of lines \cite{BruceGiblin, IsmailHamad}.
Consider a family of lines in $\R^2=\{(x, y)\}$ given by
$$F_{t}(x, y)=u(t) x+v(t)y +w(t)=0,$$
where $u(t), v(t), w(t)$ are polynomials of degree at most $2$, and $(u(t), v(t))\neq(0, 0)$ for any $t.$
Then $F_t(x, y)$ can be expressed as
$$F(t, x, y)=A(x, y)t^2+B(x,y) t+C(x, y)=0, \quad F_t(x, y)=F(t, x, y)$$
where $A, B, C$ are polynomials of degree at most $1.$

\begin{proposition}
The envelope of the family of lines $F_{t}(x, y)=0$ is given by the equation $B^2-4AC=0.$
Moreover, the envelope $B^2-4AC=0$ coincides with a quadratic curve, and the converse is also true. Here we restrict our attention to real (non-degenerate) quadratics: the ellipse, hyperbola, and parabola.
\end{proposition}

\begin{proof}
The envelope is given by the set
$$\{(x, y)\,\vert\, \text{there exists} \;\, t \;  \text{with} \;  F(t, x, y)=\frac{\partial F}{\partial t}(t, x, y)=0\}.$$

Solving $\frac{\partial F}{\partial t}=0$, we have $t=-\frac{B}{2A}$ (assuming $A\neq 0$).
Substituting this into $F=0$, we get $B^2-4AC=0.$
Since $A, B, C\in \R[x, y]$ have degree at most 1, $B^2-4AC=0$ is a quadratic equation.

Conversely, for a general quadratic equation $ax^2+bxy+cy^2+dx+ey+f=0$, if $b\neq 0$, by rotating about the origin, we may suppose that
\begin{equation}
\Bar{a}x^2+\Bar{b}y^2+\Bar{c}x+\Bar{d}y+\Bar{e}=0.
\label{eq:canonical_quad}
\end{equation}

Let $\Delta_0=\Bar{a}\Bar{b}$ and $\Delta=\Bar{a}{\Bar{d}} ^2+\Bar{b}\Bar{c}^2-4\Bar{a}\Bar{b}\Bar{e}.$

If $\Delta_0\neq 0$, by completing the square in \eqref{eq:canonical_quad}, we have
\begin{equation*}
\Bar{a}\left(x+\frac{\Bar{c}}{2\Bar{a}}\right)^2+\Bar{b}\left(y+\frac{\Bar{d}}{2\Bar{b}}\right)^2-\frac{\Delta}{4\Delta_0}=0.
\end{equation*}

Moreover, if $\Delta \neq 0$, we can rewrite this as:
\begin{equation}
\frac{\left(y+\frac{\bar{d}}{2\bar{b}}\right)^2}{\frac{\Delta}{4\bar{b}\Delta_0}}
- 4 \left( \frac{1}{4} - \frac{\left(x+\frac{\bar{c}}{2\bar{a}}\right)^2}{\frac{\Delta}{\bar{a}\Delta_0}} \right) = 0.
\label{eq:ellipse_form}
\end{equation}

In the case where $\Delta_0=0$, if specifically $\bar{a}=0$ and $\bar{b}\neq 0$, equation \eqref{eq:canonical_quad} becomes
\begin{equation}
\left(y+\frac{\bar{d}}{2\bar{b}}\right)^2 +
\left(\frac{\bar{c}}{\bar{b}} x + \frac{4\bar{b}\bar{e}-\bar{d}^2}{4\bar{b}^2}\right) = 0.
\label{eq:parabola_form}
\end{equation}
Similarly, if $\bar{a}\neq 0$ and $\bar{b}=0$, interchanging $x$ and $y$ yields an analogous equation.

For real non-degenerate quadratics, setting $A, B$, and $C$ as follows allows the equation to be expressed in the form $B^2 - 4AC = 0$.

\bigskip
\noindent
\textbf{Case 1: Ellipse ($\Delta_0 > 0$ and $\bar{a}\Delta > 0$)}

Equation \eqref{eq:canonical_quad} represents an ellipse. Based on \eqref{eq:ellipse_form}, we set
\begin{equation}
A = \frac{1}{2} - \frac{x+\frac{\bar{c}}{2\bar{a}}}{\sqrt{\frac{\Delta}{\bar{a}\Delta_0}}}, \quad
B = \frac{y+\frac{\bar{d}}{2\bar{b}}}{\sqrt{\frac{\Delta}{4\bar{b}\Delta_0}}}, \quad
C = \frac{1}{2} + \frac{x+\frac{\bar{c}}{2\bar{a}}}{\sqrt{\frac{\Delta}{\bar{a}\Delta_0}}}.
\label{eq:ABC_ellipse}
\end{equation}

\bigskip
\noindent
\textbf{Case 2: Hyperbola ($\Delta_0 < 0$ and $\Delta \neq 0$)}

If $\bar{a}\Delta < 0$, based on \eqref{eq:ellipse_form}, we set
\begin{equation}
A = \frac{x+\frac{\bar{c}}{2\bar{a}}}{\sqrt{\frac{\Delta}{\bar{a}\Delta_0}}} - \frac{1}{2}, \quad
B = \frac{y+\frac{\bar{d}}{2\bar{b}}}{\sqrt{\frac{-\Delta}{4\bar{b}\Delta_0}}}, \quad
C = \frac{x+\frac{\bar{c}}{2\bar{a}}}{\sqrt{\frac{\Delta}{\bar{a}\Delta_0}}} + \frac{1}{2}.
\label{eq:ABC_hyperbola1}
\end{equation}
If $\bar{a}\Delta > 0$, interchanging $x$ and $y$, we set
\begin{equation*}
A = \frac{y+\frac{\bar{d}}{2\bar{b}}}{\sqrt{\frac{\Delta}{\bar{b}\Delta_0}}} - \frac{1}{2}, \quad
B = \frac{x+\frac{\bar{c}}{2\bar{a}}}{\sqrt{\frac{-\Delta}{4\bar{a}\Delta_0}}}, \quad
C = \frac{y+\frac{\bar{d}}{2\bar{b}}}{\sqrt{\frac{\Delta}{\bar{b}\Delta_0}}} + \frac{1}{2}.
\end{equation*}

\bigskip
\noindent
\textbf{Case 3: Parabola ($\Delta_0 = 0$ and $\Delta \neq 0$)}

If $\bar{a} = 0$, then $\bar{b} \neq 0$ and $\bar{c} \neq 0$. From \eqref{eq:parabola_form}, we set
\begin{equation}
A = 1, \quad B = y + \frac{\bar{d}}{2\bar{b}}, \quad
C = -\frac{1}{4} \left( \frac{\bar{c}}{\bar{b}} x + \frac{4\bar{b}\bar{e}-\bar{d}^2}{4\bar{b}^2} \right).
\label{eq:ABC_parabola1}
\end{equation}
If $\bar{b} = 0$, then $\bar{a} \neq 0$ and $\bar{d} \neq 0$. Interchanging $x$ and $y$, we set
\begin{equation*}
A = 1, \quad B = x + \frac{\bar{c}}{2\bar{a}}, \quad
C = -\frac{1}{4} \left( \frac{\bar{d}}{\bar{a}} y + \frac{4\bar{a}\bar{e}-\bar{c}^2}{4\bar{a}^2} \right).
\end{equation*}
\end{proof}

\section{Main Theorem}
Although the definitions above were stated on $\R^2$ for simplicity, they apply equally to any open subset of $\R^2$.
Let $Q$ be a non-degenerate real quadratic curve, and let $\Omega_Q$ be the open set of points from which two distinct real tangent lines to $Q$ can be drawn. On $\Omega_Q$, these tangent lines define two transverse foliations, hence a Lagrangian 2-web.

\begin{theorem}
Let $Q$ be a non-degenerate real quadratic curve. Then the Lagrangian 2-web defined on $\Omega_Q$ by the tangent lines to $Q$ is not a Samuelson web.
\end{theorem}

\begin{proof}
First, we note that any non-degenerate real quadratic curve can be transformed into one of the canonical forms by a composition of translations and rotations. Since translations and rotations preserve the standard symplectic form $\omega=dx\wedge dy$, and since the definition of a Samuelson web is invariant under symplectic equivalence, it suffices to prove the statement for the canonical forms.

By Remark 2.2, it suffices to show that for each conic, the d'Alembert condition is not satisfied; that is, the quantity $K$ defined in Remark 2.4 does not vanish. 

\bigskip
\noindent
$\bullet$ {\bf Ellipse}

Consider an ellipse with the canonical form $\dfrac{x^2}{p^2} + \dfrac{y^2}{q^2} = 1$.
By \eqref{eq:ABC_ellipse}, the family of its tangent lines is given by
\begin{equation*}
F(t, x, y) = \left(\frac{1}{2} - \frac{x}{2p}\right)t^2 + \frac{y}{q}t + \frac{1}{2} + \frac{x}{2p} = 0.
\end{equation*}
Solving for $t$, we get
\begin{equation*}
t = \frac{-\frac{y}{q} \pm \sqrt{\frac{x^2}{p^2} + \frac{y^2}{q^2} - 1}}{1 - \frac{x}{p}}.
\end{equation*}
These solutions yield the level functions $u(x, y)$ and $v(x, y)$, whose level curves coincide with the lines $F_t(x,y)=0$.
The slope $a(x, y) = -u_x/u_y$ is given by
\begin{equation*}
a(x, y) = \frac{-xy + \sqrt{q^2x^2 + p^2y^2 - p^2q^2}}{p^2 - x^2}.
\end{equation*}
Similarly,
\begin{equation*}
b(x, y) = \frac{-xy - \sqrt{q^2x^2 + p^2y^2 - p^2q^2}}{p^2 - x^2}.
\end{equation*}
A direct computation, assisted by Mathematica, yields
\begin{equation*}
K = \frac{4(3q^2x^2 + 3p^2y^2 - 4p^2q^2)}{(p^2 - x^2)^3}.
\end{equation*}
Since $K$ is not identically zero, the condition is not satisfied.

\bigskip
\noindent
$\bullet$ {\bf Hyperbola}

Consider a hyperbola with the canonical form $\dfrac{x^2}{p^2} - \dfrac{y^2}{q^2} = 1$.
By \eqref{eq:ABC_hyperbola1}, we have
\begin{equation*}
F(t, x, y) = \left(\frac{x}{2p} - \frac{1}{2}\right)t^2 + \frac{y}{q}t + \frac{x}{2p} + \frac{1}{2} = 0.
\end{equation*}
Solving for $t$, we get
\begin{equation*}
t = \frac{-\frac{y}{q} \pm \sqrt{1 - \left(\frac{x^2}{p^2} - \frac{y^2}{q^2}\right)}}{-1 + \frac{x}{p}}.
\end{equation*}
The slopes $a(x, y)$ and $b(x, y)$ are given by
\begin{equation*}
a(x, y), b(x, y) = \frac{-xy \pm \sqrt{p^2q^2 - (q^2x^2 - p^2y^2)}}{p^2 - x^2}.
\end{equation*}
A direct computation, assisted by Mathematica, yields
\begin{equation*}
K = \frac{4(4p^2q^2 - (3q^2x^2 - 3p^2y^2))}{(p^2 - x^2)^3} \neq 0.
\end{equation*}

\bigskip
\noindent
$\bullet$ {\bf Parabola}

Consider a parabola with the canonical form $y^2 = 4px$.
By \eqref{eq:ABC_parabola1}, we have
\begin{equation*}
F(t, x, y) = t^2 + yt + px = 0.
\end{equation*}
Solving for $t$, we get
\begin{equation*}
t = \frac{-y \pm \sqrt{y^2 - 4px}}{2}.
\end{equation*}
The slopes $a(x, y)$ and $b(x, y)$ are given by
\begin{equation*}
a(x, y), b(x, y) = \frac{y \pm \sqrt{y^2 - 4px}}{2x}.
\end{equation*}
A direct computation, assisted by Mathematica, yields
\begin{equation*}
K = \frac{-p}{x^3} \neq 0 \quad (\text{since } p \neq 0).
\end{equation*}
This completes the proof.
\end{proof}

\end{document}